\documentclass{amsart}
\usepackage{setspace, amsmath, amsthm, amssymb, amsfonts, amscd, epic, graphicx, ulem, dsfont}
\usepackage[T1]{fontenc}
\usepackage{multirow}
\usepackage{bbm}
\usepackage{enumerate}

\makeatletter \@namedef{subjclassname@2010}{
  \textup{2010} Mathematics Subject Classification}
\makeatother

\newtheorem{thm}{Theorem}[section]
\newtheorem{cor}[thm]{Corollary}
\newtheorem{lem}[thm]{Lemma}
\newtheorem{pro}[thm]{Proposition}

\theoremstyle{remark}
\newtheorem*{rema}{Remark}

\theoremstyle{definition}
\newtheorem*{defn}{Definition}
\newtheorem{exa}[thm]{\textbf{Example}}

\newcommand{\Ima}{\text{\rm{Im}}}
\newcommand{\REAL}{\text{\rm{Re}}}

\newcommand{\R}{\mathbb{R}}

\newcommand{\N}{\mathbb{N}}

\newcommand{\C}{\mathbb{C}}

\begin{document}

\title[Invertibility of Sums]{On the Invertibility of the Sum of Operators}
\author[M. H. MORTAD]{Mohammed Hichem Mortad}

\dedicatory{}
\thanks{}
\date{}


 \address{ Department of
Mathematics, University of Oran 1, Ahmed Ben Bella, B.P. 1524, El
Menouar, Oran 31000, Algeria.\newline {\bf Mailing address}:
\newline Pr Mohammed Hichem Mortad \newline BP 7085 Seddikia Oran
\newline 31013 \newline Algeria}

\email{mhmortad@gmail.com, mortad.hichem@univ-oran1.dz.}

\begin{abstract}
The primary purpose of this paper is to investigate the question of
invertibility of the sum of operators. The setting is bounded and
unbounded linear operators. Some interesting examples and
consequences are given. As an illustrative point, we characterize
invertibility for the class of normal operators. Also, we give a
very short proof of the self-adjointness of a normal operator when
the latter has a real spectrum.
\end{abstract}

\maketitle

\section{Introduction}

Let $H$ be a complex Hilbert space and let $B(H)$ denote the algebra
of all bounded linear operators on $H$.  An $A\in B(H)$ is called
positive (symbolically $A\geq0$) if
\[\langle Ax,x\rangle\geq0~\forall x\in H.\]
By a square root of $A\in B(H)$, we mean a $B\in B(H)$ such that
$B^2=A$. If $A\geq0$, then there is one and only one $B\geq0$ such
that $B^2=A$. This positive $B$ is denoted by $\sqrt A$.

Recall that any $T\in B(H)$ is expressible as $T=A+iB$ where $A,B\in
B(H)$ are self-adjoint. Besides,
\[A=\REAL T=\frac{T+T^*}{2} \text{ and } B=\Ima T=\frac{T-T^*}{2i}.\]
It is readily verifiable that $T$ is normal iff $AB=BA$.

We also recall some known results which will be called on below
(these are standard facts, see \cite{Mortad-Oper-TH-BOOK-WSPC} for
proofs).

\begin{thm}\label{theorem squrae root first} Let $A,B\in B(H)$. Then
\[0\leq A\leq B\Longrightarrow \sqrt A\leq \sqrt B.\]
\end{thm}

\begin{lem}\label{sqrt A+B Lemma}
Let $A,B\in B(H)$ be such that $AB=BA$ and $A,B\geq 0$. Then
$\sqrt{A+B}\leq \sqrt A+\sqrt B.$
\end{lem}

Theorem \ref{theorem squrae root first} is known to hold for
$\alpha\in (0,1)$ instead of $\frac{1}{2}$ (the Heinz Inequality).
Hence we may easily establish the analogue of Lemma \ref{sqrt A+B
Lemma} for $n$th roots.

\begin{lem}
Let $A,B\in B(H)$ be such that $AB=BA$ and $A,B\geq 0$. If $n\in\N$,
then $(A+B)^{\frac{1}{n}}\leq A^{\frac{1}{n}}+B^{\frac{1}{n}}$.
\end{lem}

Let us say a few more words about the absolute value of an operator
(that is, $|A|=\sqrt{A^*A}$ with $A\in B(H)$). It is well known that
the properties of the absolute value of complex numbers cannot all
just be carried over to $B(H)$ (even for self-adjoint operators).
This applies for example to the multiplicativity property and to
triangle inequalities. For counterexamples, readers may wish to
consult \cite{Mortad-Oper-TH-BOOK-WSPC}. See also
\cite{mortad-Abs-value-BD} to see when these results hold. The
similar question on unbounded operators may be found in the recent
work \cite{boucif-Dehimi-Mortad-ABS-VALUE}. Some results, however,
do hold without any special assumption. One of them is the following
simple result.

\begin{pro}\label{absol val ineq no assump PRO} Let $A,B\in B(H)$. Then
\[|A+B|^2\leq 2|A|^2+2|B|^2.\]
\end{pro}

The following known result is also primordial.

\begin{pro}\label{invert positivity A B New Deh-Mortad PRO}(see \cite{Dehimi-Mortad-Reid} for a new proof)
Let $A,B\in B(H)$. If $0\leq A\leq B$ and if $A$ is invertible, then
$B$ is invertible and $B^{-1}\leq A^{-1}$.
\end{pro}

We digress a little bit to notice a simple proof of the positiveness
of the spectrum of a positive operator using the previous
proposition: \textit{If $\lambda<0$, then $-\lambda I>0$ and so
$A-\lambda I\geq -\lambda I$ because $A\geq0$. Hence $A-\lambda I$
is invertible as $-\lambda I$ is, i.e. $\lambda\not\in\sigma(A)$.}

A simple application of the Functional Calculus for self-adjoint
operators is as follows.

\begin{exa}\label{A alpha geq A alpha smaller than one EXA} Let $A\in B(H)$ be such that $0\leq A\leq I$. If
$\alpha\in [0,1]$, then $A^{\alpha}\geq A$.
\end{exa}

From \cite{Kadison-Ringrose} we recall the following result.

\begin{pro}\label{spectra prod sum commut PRO}
Let $A,B\in B(H)$ be commuting. Then
\[\sigma(A+B)\subset \sigma(A)+\sigma(B)\]
where
\[\sigma(A)+\sigma(B)=\{\lambda+\mu:~\lambda\in \sigma(A),\mu\in\sigma(B)\}.\]
\end{pro}

We call the result in the previous proposition the "subadditivity of
the spectrum". There is also a "submultiplicativity of the
spectrum", that is,

\begin{pro}(\cite{Kadison-Ringrose}) \label{spectra PROD commut PRO}
Let $A,B\in B(H)$ be commuting. Then
\[\sigma(AB)\subset \sigma(A)\sigma(B)\]
where
\[\sigma(A)\sigma(B)=\{\lambda\mu:~\lambda\in \sigma(A),\mu\in\sigma(B)\}.\]
\end{pro}

In Proposition \ref{spectra prod sum equivalence PRO}, we show that
Proposition \ref{spectra PROD commut PRO} implies Proposition
\ref{spectra prod sum commut PRO} in the context of self-adjoint
operators and that the backward implication also holds but for
positive and invertible operators.

Recall also the following definition.

\begin{defn}\label{defn pos unbd schmud}
Let $T$ and $S$ be unbounded positive self-adjoint operators. We say
that $S\geq T$ if $D(S^{\frac{1}{2}})\subseteq D(T^{\frac{1}{2}})$ and $%
\left\| S^{\frac{1}{2}}x\right\| \geq \left\| T^{\frac{1}{2}%
}x\right\|$ for all $x\in D(S^{\frac{1}{2}}).$
\end{defn}

The "natural but weak extension" is defined as in Definition 10.5
(Page 230) in \cite{SCHMUDG-book-2012}: If $S$ and $T$ are
non-necessarily bounded symmetric operators, then  $S\succeq T$ if
$D(S)\subset D(T)$ and
\[\langle Sx,x\rangle\geq \langle Tx,x\rangle~\forall x\in D(S).\]

Notice that Proposition \ref{invert positivity A B New Deh-Mortad
PRO} remains valid for unbounded operators. Indeed, as on Page 200
in \cite{WEI}, if $S$ and $T$ are self-adjoint, $T$ is boundedly
invertible and $S\geq T\geq0$, then $S$ is boundedly invertible and
$S^{-1}\leq T^{-1}$.

Finally, we assume that readers are familiar with other basic
notions and results of Operator Theory.

The sum of two invertible operators is not necessarily invertible
even if strong conditions are imposed. For instance, if we take $A$
to be invertible and positive, then setting $B=-A$, we see that
$AB=BA$ and that $B$ is invertible. But plainly $A+B$ is not
invertible. Positivity must also be avoided as it may make some of
the results evident. For instance, if $A,B\in B(H)$ are such that
$A,B\geq 0$ and $A$ say is invertible, then obviously $A+B$ ($\geq
A$) is invertible by Proposition \ref{invert positivity A B New
Deh-Mortad PRO}. These two observations make the investigation of
this question a little hopeless. However, the approach considered by
Bikchentaev in \cite{Bikchentaev-invert-sum} deserves to be
investigated further. This is one aim of the paper. Another purpose
is to treat some of these questions in an unbounded setting. Some
interesting consequences, examples and counterexamples accompany our
results.

\section{Main Results}

\begin{thm}\label{invertibility sum abs valu BITChk THM}
Let $A,B\in B(H)$.
\begin{enumerate}
  \item If $A+B$ is invertible, then so is $|A|^2+|B|^2$.
  \item If $A+B$ is invertible, then so are $|A|+|B|$ and $|A|^{2^n}+|B|^{2^n}$ ($n\in\N$)  as
  well.
  \item Assume here that $A,B\geq 0$ and let $\alpha,\beta\in\C$.
  Then
  \[\alpha A+\beta B\text{ invertible }\Longrightarrow A+B \text{ invertible}.\]
\end{enumerate}
\begin{rema}
Most of the previous results appeared in
\cite{Bikchentaev-invert-sum}, but our proof is simpler.
\end{rema}
\end{thm}

\begin{proof} (1): It is clear that
  \[A+B \text{ invertible }
  \Longrightarrow |A+B|^2 \text{ invertible }
  \Longrightarrow |A|^2+|B|^2 \text{ invertible}\]
  by Propositions \ref{absol val ineq no assump PRO} \& \ref{invert positivity A B New Deh-Mortad PRO}.

(2): By the first property, $|A|^2+|B|^2$ is
    invertible from which we readily get that $|A|^4+|B|^4$ is
    invertible and, by induction, we establish the invertibility of
    $|A|^{2^n}+|B|^{2^n}$.

    The invertibility of $|A|+|B|$ is not hard to prove.
    WLOG we may assume that $\|A\|\leq 1$ and $\|B\|\leq 1$. Hence
    $A^*A\leq I$ and $B^*B\leq I$,
and so
    $|A|\leq I$ and $|B|\leq I$
    by Theorem \ref{theorem squrae root first}. By the Functional
    Calculus,
    $|A|^2\leq |A|$  and $|B|^2\leq |B|$.
    This implies that
    \[|A|+|B|\geq |A|^2+|B|^2\geq 0.\]
   Proposition \ref{invert positivity A B New Deh-Mortad PRO} allows us to confirm the invertibility of $|A|+|B|$, as
   desired.

  (3): Since $\alpha A+\beta B$ is invertible, by the previous
  property, so is $|\alpha| |A|+|\beta| |B|$ or merely $|\alpha|
  A+|\beta| B$ as $A,B\geq 0$. Since we can assume $|\alpha|\geq
  |\beta|>0$, we infer that
  \[|\alpha|(A+B)\geq |\alpha|
  A+|\beta| B.\]
  Consequently, $|\alpha|(A+B)$ or simply $A+B$
  is invertible.
\end{proof}

\begin{cor}
Let $A\in B(H)$ be invertible. Then $|A-B|+|B|$ is invertible for
every $B\in B(H)$.
\end{cor}

\begin{proof}
Since $A$ is invertible and $A=(A-B)+B$, it follows that $|A-B|+|B|$
too is invertible by the previous result.
\end{proof}

\begin{rema}
It is clear that if $T$ is invertible, then the self-adjoint $\REAL
T+\Ima T$ need not be invertible. For instance:
\end{rema}

\begin{exa}
Let $A$ be self-adjoint and invertible and set $B=-A$. Then
\[A+iB=A-iA=(1-i)A\]
is invertible while $A+B=0$ is not.
\end{exa}

Nonetheless, we have the following.

\begin{cor}
Let $A\in B(H)$ be invertible. Then $|\REAL A|+|\Ima A|$ is
invertible.
\end{cor}

\begin{proof}
Just write $A=\REAL A+i\Ima A$, then apply Theorem
\ref{invertibility sum abs valu BITChk THM}.
\end{proof}

The next corollary appeared in \cite{Bikchentaev-invert-sum}.

\begin{cor}
Let $A,B\in B(H)$ be such that $A+B$ is invertible. If $p,q\in
(0,\infty)$, then $|A|^p+|B|^q$ is invertible.
\end{cor}

\begin{proof}
WLOG, we assume that $|A|\leq I$ and $|B|\leq I$. We can always find
an $n\in\N$ such that $p,q\leq 2^n$. Then by Example \ref{A alpha
geq A alpha smaller than one EXA}, we have
\[|A|^{\frac{p}{2^n}}\geq |A| \text{ and } |B|^{\frac{q}{2^n}}\geq |B|.\]
Therefore,
\[|A|^p\geq |A|^{{2^n}} \text{ and } |B|^q\geq |B|^{{2^n}}.\]
Thus,
\[|A|^p+|B|^q\geq |A|^{{2^n}}+|B|^{{2^n}}.\]
 Since $|A|^{2^n}+|B|^{2^n}$ is already invertible (Theorem
\ref{invertibility sum abs valu BITChk THM}), we obtain the
invertibility of $|A|^p+|B|^q$ from Proposition \ref{invert
positivity A B New Deh-Mortad PRO}.
\end{proof}

Now, we give an interesting consequence on the invertibility of
matrices of bounded operators.

\begin{cor}
Let $A,B,C,D\in B(H)$ and define $T\in B(H\oplus H)$ by
\[T=\left(
      \begin{array}{cc}
        A & B \\
        C & D \\
      \end{array}
    \right).
\]
If $T$ is invertible, then so are $|A|+|C|$ and $|B|+|D|$.

In particular, if $D$ is normal and $B=0$ (resp. if $A$ is normal
and $C=0$), then
\[\sigma(D)~(resp.~ \sigma(A))\subset \sigma(T).\]
\end{cor}

\begin{proof}Write
\[T=\underbrace{\left(
      \begin{array}{cc}
        A & 0 \\
        0 & D \\
      \end{array}
    \right)}_{S}+\underbrace{\left(
      \begin{array}{cc}
        0 & B \\
        C & 0 \\
      \end{array}
    \right)}_{R}.\]
Since $T$ is invertible, so is $|S|+|R|$. But
\[|S|=\left(
      \begin{array}{cc}
        |A| & 0 \\
        0 & |D| \\
      \end{array}
    \right)\text{ and }|R|=\left(
      \begin{array}{cc}
        |C| & 0 \\
        0 & |B| \\
      \end{array}
    \right)\]
    and so
    \[\left(
      \begin{array}{cc}
        |A|+|C| & 0 \\
        0 & |B|+|D| \\
      \end{array}
    \right)\]
    is invertible. This means that $|A|+|C|$ and $|B|+|D|$ are
    invertible.

    For the last claim just reason contrapositively by remembering
    that a normal operator is invertible iff its absolute value is.
\end{proof}

\begin{pro}\label{mmmmmmm}
Let $A,B\in B(H)$ be such that $AB=BA$ and either $A$ or $B$ is
normal. Then
\[|A|+|B| \text{ is invertible} \Longleftrightarrow |A|^2+|B|^2 \text{ is invertible}.\]
\end{pro}

\begin{proof}We already proved the implication "$\Rightarrow$" in Theorem
\ref{invertibility sum abs valu BITChk THM}.

Assume now that $|A|^2+|B|^2$ is invertible. Since $A$ commutes with
$B$, it follows by the Fuglede Theorem that $A^*B=BA^*$. Hence by
known techniques, we may get the commutativity of $|A|^2$ and
$|B|^2$. Therefore, from Lemma \ref{sqrt A+B Lemma} we obtain
\[\sqrt {|A|^2}+\sqrt{|B|^2}\geq \sqrt{|A|^2+|B|^2}.\]
Hence
\[|A|+|B|\geq \sqrt{|A|^2+|B|^2}.\]
Since $\sqrt{|A|^2+|B|^2}$ is invertible, Proposition \ref{invert
positivity A B New Deh-Mortad PRO} gives the invertibility of
$|A|+|B|$, as wished.
\end{proof}

\begin{rema}
The invertibility of $A^2+B^2$ does not yield that of $A+B$ even in
the event of the self-adjointness of $A$ and $B$. As a
counterexample, just consider an invertible and self-adjoint $A$
such that $A=-B$.
\end{rema}

\begin{exa}
Let $A\in B(H)$. We know that $\cos^2A+\sin^2A=I$.  It then follows that $|\cos A|+|\sin A|$ is invertible.
\end{exa}

\begin{exa}
Let $A$ be self-adjoint and invertible. Set $B=iA$. Then $A^2+B^2=0$
is obviously not invertible.
\end{exa}

The idea of the proof of Proposition \ref{mmmmmmm} leads to the
following generalization.

\begin{pro}
Let $A,B\in B(H)$ be such that $AB=BA$ and either $A$ or $B$ is
normal. Then
\[|A|+|B| \text{ is invertible} \Longleftrightarrow |A|^n+|B|^n \text{ is invertible}\]
where $n\in\N$.
\end{pro}

The following example shows that both the real and imaginary parts
of an invertible operator may be not invertible.

\begin{exa}
Consider the multiplication operators
\[Af(x)=\cos xf(x) \text{ and } Bf(x)=\sin xf(x)\]
both defined on $L^2(\R)$. Then $A$ and $B$ are not invertible.
However,
\[Tf(x)=(A+iB)f(x)=(\cos x+i\sin x)f(x)=e^{ix}f(x)\]
is invertible (even unitary!).
\end{exa}

It is fairly easy to see that a \textit{normal} $T=A+iB$ is
invertible iff $A^2+B^2$ is invertible (see e.g.
\cite{Mortad-Oper-TH-BOOK-WSPC}). With this observation, we may
state the following interesting characterization of invertibility
for the class of normal operators.

\begin{pro}\label{normal invertible sum COR}
Let $T=A+iB$ be normal in $B(H)$. Then
\[T\text{ is invertible}\Longleftrightarrow |A|+|B| \text{ is invertible.}\]
In particular, if $\lambda=\alpha+i\beta$, then
\[\lambda\in\sigma(T)\Longleftrightarrow |A-\alpha I|+|B-\beta I| \text{ is not invertible.}\]
\end{pro}

\begin{rema}(cf. \cite{boucif-Dehimi-Mortad-ABS-VALUE})
Another way of establishing the previous result is as follows. By
\cite{mortad-Abs-value-BD}, we know that if $T\in B(H)$ is normal,
then
$|T|\leq |\REAL T|+|\Ima T|$.
Hence the invertibility of $T$ entails that $|\REAL T|+|\Ima T|$.
Conversely, if $T$ is normal (in fact hyponormality suffices here),
then
$|\Ima T|\leq |T|$ and  $|\REAL T|\leq |T|$
and so
\[|\Ima T|+|\REAL T|\leq 2|T|.\]
Therefore, the invertibility of $|\Ima T|+|\REAL T|$ implies that of
$T$.
\end{rema}

The following related version to Proposition \ref{spectra prod sum
commut PRO} does not make use of the Gelfand Transform.

\begin{pro}\label{sepctrum sum normal NEW PRO}
Let $S,T\in B(H)$ be normal and such that $ST=TS$. Then
\[\sigma(S+T)\subset \sigma(\REAL S+\REAL T)+i\sigma(\Ima S+\Ima T).\]
\end{pro}

\begin{proof}
Write $S=A+iB$ and $T=C+iD$. Since $ST=TS$ and $S$ and $T$ are
normal, $S+T$ is normal (see e.g. \cite{RUD}). Hence, if we let
$\lambda=\alpha+i\beta\in\C$, then
\[S+T-\lambda I=(A+C-\alpha I)+i(B+D-\beta I)\]
becomes normal. So, if $\lambda\in\sigma(S+T)$, then Proposition
\ref{normal invertible sum COR} says that $|A+C-\alpha I|+|B+D-\beta
I|$ is not invertible. If either $|A+C-\alpha I|$ or $|B+D-\beta I|$
is invertible, then clearly $|A+C-\alpha I|+|B+D-\beta I|$ would be
invertible! Therefore, both $|A+C-\alpha I|$ and $|B+D-\beta I|$ are
not invertible, i.e. $A+C-\alpha I$ and $B+D-\beta I$ are not
invertible. In other language, $\alpha\in\sigma(A+C)$ and
$\beta\in\sigma(B+D)$. Accordingly,
$\lambda=\alpha+i\beta\in\sigma(A+C)+i\sigma(B+D)$, as wished.
\end{proof}

\begin{cor}Let $T=A+iB$ be normal in $B(H)$. Then
\[\sigma(T)\subset \sigma(A)+i\sigma(B).\]
\end{cor}

As another consequence, we have a new and shorter proof of a well
known result.

\begin{cor}\label{s.a. real spectrum cor}
Let $A\in B(H)$ be self-adjoint. Then $\sigma(A)\subset\R$.
\end{cor}

\begin{proof}
Let $\lambda\not\in\R$, i.e. $\lambda=\alpha+i\beta$ ($\alpha\in\R$,
$\beta\in\R^*$). Since $A-\alpha I$ is self-adjoint, it follows that
$A-\alpha I-i\beta I$ is normal. By the invertibility of $|\beta|I$,
it follows that of $|A-\alpha I|+|\beta I|$ (by Proposition
\ref{invert positivity A B New Deh-Mortad PRO}). By Proposition
\ref{normal invertible sum COR}, this means that $A-\lambda I$ is
invertible, that is, $\lambda\not\in \sigma(A)$.
\end{proof}

The following result appeared in \cite{mortad-Abs-value-BD}.

\begin{pro}
Let $A,B\in B(H)$ be such that $AB=BA$. If $A$ is normal and $B$ is
hyponormal, then the following inequality holds:
\[||A|-|B||\leq |A-B|.\]
\end{pro}

As a consequence of the previous result, we get a very short proof
concerning the spectrum of unitary operators.

\begin{cor}\label{unitary unite cercle spectrum cor}
Let $A\in B(H)$ be unitary. Then $\sigma(A)\subset
\{\lambda\in\C:~|\lambda|=1\}$.
\end{cor}

\begin{proof}
We have $|A|=I$ and so
$|1-|\lambda||I=|I-|\lambda|I|\leq |A-\lambda I|$.
Thus, if $|\lambda|\neq1$, then $\lambda\not\in\sigma(A)$.
\end{proof}

It is known that a normal operator having a real spectrum is
self-adjoint. The proof is very simple if we know the very
complicated Spectral Theorem for normal operators. It would be
interesting to prove this result along the lines of the proof of
Corollary \ref{s.a. real spectrum cor}. Notice also that this can
very easily be established if the imaginary part of $T$ is a scalar
operator. A new proof in the general case has not been obtained yet.
Nonetheless, as an application of Proposition \ref{spectra prod sum
commut PRO}, we have the following short proof (which seems to have
escaped notice) of this result.

\begin{pro}\label{normal real spectrum s.a. PROPO}
Let $T=A+iB$ be normal in $B(H)$ and such that $\sigma(T)\subset\R$.
Then $T$ is self-adjoint.
\end{pro}

\begin{proof}Recall that $A$ and $B$ are self-adjoint.
The normality of $T$ is equivalent to $AB=BA$. Hence $TA=AT$. Since
$T-A=iB$, we have by Proposition \ref{spectra prod sum commut PRO}
\[i\sigma(B)=\sigma(iB)=\sigma(T-A)\subset\sigma(T)+\sigma(-A)\subset \R.\]
Thus, necessarily $\sigma(B)=\{0\}$. Accordingly, the Spectral
Radius Theorem gives us $B=0$  and so $T=A$, i.e. $T$ is
self-adjoint.
\end{proof}

What is also interesting is that since Proposition \ref{spectra prod
sum commut PRO} holds in the context of Banach algebras (see Theorem
11.23 in \cite{RUD}), Proposition \ref{normal real spectrum s.a.
PROPO} becomes valid in the context of $C^*$-algebras. The proof is
identical and so it is omitted.

\begin{pro}Let $\mathcal{A}$ be a $C^*$-algebra and let $a\in
\mathcal{A}$ be a normal element. If $\sigma(a)\subset\R$, then $a$
is hermitian.
\end{pro}

We also have the following related result.

\begin{pro}
Let $\mathcal{A}$ be a $C^*$-algebra and let $a\in \mathcal{A}$ be a
normal element. If $\sigma(a)$ is purely imaginary, then $a$ is
skew-hermitian (i.e. $a^*=-a$).
\end{pro}

\begin{proof}Write $a=x+iy$ where $x$ and $y$ are commuting hermitian elements
of $\mathcal{A}$. Write $x=a-iy$ and proceed as above to force
$x=0$.
\end{proof}

The last result about the spectrum is the following.

\begin{pro}\label{spectra prod sum equivalence PRO}
The "subadditivity of the spectrum" is equivalent to the
"submultiplicativity of the spectrum" in the context of bounded
positive, commuting and invertible operators.
\end{pro}

\begin{proof}
Let $A,B\in B(H)$ be self-adjoint and such that $AB=BA$. Assume that
Proposition \ref{spectra PROD commut PRO} holds. Let $\lambda\in
\sigma(A+B)$. Then the Spectral Mapping Theorem yields
\[e^{\lambda}\in \sigma(e^{A+B})=\sigma(e^{A}e^{B})\subset
\sigma(e^{A})\sigma(e^{B}),\] that is,
$e^{\lambda}=e^{\alpha}e^{\beta}$ for $\alpha\in\sigma(A)$ and
$\beta\in\sigma(B)$. Hence
\[\lambda=\ln(e^{\alpha}e^{\beta})=\alpha+\beta\in \sigma(A)+\sigma(B).\]

Now, suppose that Proposition \ref{spectra prod sum commut PRO}
holds. Suppose also here that $A$ and $B$ are positive and
invertible. Let $\lambda \in \sigma(AB)$. Since $AB$ is positive and
invertible, by the Spectral Mapping Theorem we get
\[\ln \lambda\in \sigma[\ln (AB)]=\sigma(\ln A+\ln B)\subset \sigma(\ln A)+\sigma(\ln B)=\ln[\sigma (A)]+\ln [\sigma (B)],\]
i.e. $\ln \lambda=\ln {\alpha}+\ln{\beta}$ with $\alpha\in\sigma(A)$
and $\beta\in\sigma(B)$. Thus, $\lambda=\alpha\beta \in
\sigma(A)\sigma(B)$, as required.
\end{proof}

Let's go back to invertibility.

\begin{pro}\label{jjj}
Let $A,B\in B(H)$ be such that $AB\geq 0$ and $AB$ is invertible.
Then $|A^*|^2+|B|^2$ is invertible.
\end{pro}

\begin{proof}Let $x\in H$. Then
\begin{align*}
&\langle ABx,x\rangle=|\langle ABx,x\rangle|=|\langle Bx,A^*x\rangle|\\
&\leq \|A^*x\|\|Bx\|\leq 2(\|A^*x\|^2+\|Bx\|^2)\\
&=2(\langle AA^*x,x\rangle+\langle
B^*Bx,x\rangle)=\langle2(AA^*+B^*B)x,x\rangle,
\end{align*}
that is,
$AB\leq 2(AA^*+B^*B)=2(|A^*|^2+|B|^2)$.
Proposition \ref{invert positivity A B New Deh-Mortad PRO} allows us
to establish the invertibility of $|A^*|^2+|B|^2$, as required.
\end{proof}

\begin{rema}
Notice that the result is obvious if either $A$ or $B$ is
invertible. In order to keep the result non-trivial we need also to
avoid $\ker A=\ker A^*$ \textit{and} $\ker B=\ker B^*$ (see
\cite{Dehimi-Mortad-2018}). This remark applies to the unbounded
case as well (treated in Theorem \ref{UNBD MAIN THM} below).
\end{rema}

\begin{rema}
The fact that we have assumed the invertibility of $AB$ is essential
as seen by the following example.
\end{rema}

\begin{exa}
Let $A$ be the positive operator defined by
\[Af(x)=xf(x),~f\in L^2(0,1).\]
Setting $B=A$, we see that
$ABf(x)=x^2f(x)$
is positive. However,
\[(|A^*|^2+|B|^2)f(x)=(A^2+B^2)f(x)=2x^2f(x)\]
is not invertible.
\end{exa}

As a consequence of the foregoing proposition, we have the
following.

\begin{cor}
Let $A\in B(H)$ be right (resp. left) invertible with $B\in B(H)$
being its right (resp. left) inverse. Then $|A^*|^2+|B|^2$ (resp.
$|A|^2+|B^*|^2$) is always invertible.
\end{cor}

\begin{proof}
Since $A$ is right invertible, for some $B\in B(H)$, we have $AB=I$.
Since the latter is positive and invertible, the result follows
immediately. The case of left-invertibility is identical.
\end{proof}

We can generalize Proposition \ref{jjj} to unbounded operators.

\begin{thm}\label{UNBD MAIN THM}
Let $A$ and $B$ be two operators such that $A$ is closed and $B\in
B(H)$.
\begin{enumerate}
  \item If $BA$ is positive (i.e. $\langle BAx,x\rangle\geq0$ for all $x\in D(BA)$) and invertible, then $|A|^2+|B^*|^2$ is
  invertible. Besides,
  \[\langle(|A|^2+|B^*|^2)^{-1}x,x\rangle\leq \langle(BA)^{-1}x,x\rangle~\forall x\in H.\]
  \item If $AB$ is densely defined, positive and invertible, then $|A^*|^2+|B|^2$ is
  invertible. Moreover,
  \[\langle(|A^*|^2+|B|^2)^{-1}x,x\rangle\leq \langle(AB)^{-1}x,x\rangle~\forall x\in H.\]
\end{enumerate}
\end{thm}

\begin{proof} (1): The first step is to show that
  $BA\leq 2|A|^2+2|B^*|^2$.
  Observe that
  \[D(|A|^2+|B^*|^2)=D(A^*A+BB^*)=D(A^*A)\subset D(A)=D(BA).\]
  Now, let $x\in D(A^*A)$. As in the bounded case, we may prove
  \[\langle BAx,x\rangle=|\langle BAx,x\rangle|=|\langle Ax,B^*x\rangle|\leq \|Ax\|\|B^*x\|\leq 2\langle(|A|^2+|B^*|^2)x,x\rangle.\]
  This means that $BA\preceq 2|A|^2+2|B^*|^2$. Since $BA$ is positive, it is
(only) symmetric. Since it is invertible, it follows that $BA$ is
actually self-adjoint (and positive). Thus, by Lemma 10.10 in
\cite{SCHMUDG-book-2012}, "$\preceq$" becomes "$\leq $", that is, we
have established the desired inequality
$BA\leq 2|A|^2+2|B^*|^2$.
Since $BA$ is positive, invertible and $BA$ and $|A|^2+|B^*|^2$ are
self-adjoint, it follows that $|A|^2+|B^*|^2$ is invertible by the
unbounded version of Proposition \ref{invert positivity A B New
Deh-Mortad PRO} (recalled also in the introduction), as wished.

(2): The idea is similar to the previous case. As $AB$ is symmetric
  and invertible, then it is self-adjoint (and positive). Hence
  $B^*A^*\subset (AB)^*=AB$
  and so $D(A^*)=D(B^*A^*)\subset D(AB)$. The main point is to show that
  \[AB\leq 2|A^*|^2+2|B|^2.\]
  Clearly,
  \[D(|A^*|^2+|B|^2)=D(|A^*|^2)=D(AA^*)\subset D(A^*)\subset D(AB).\]
  Now, let $x\in D(AA^*)$. Then
  \begin{eqnarray*}
  & & \langle ABx,x\rangle =\langle Bx,A^*x\rangle
  \leq \|Bx\|\|A^*x\|\\
  & & \leq 2\|Bx\|^2+\|A^*x\|^2
=2\langle Bx,Bx\rangle +2\langle A^*x,A^*x\rangle
=\langle2(|B|^2+|A^*|^2)x,x\rangle.
  \end{eqnarray*}
  As above, $AB\preceq 2|B|^2+2|A^*|^2$ becomes $AB\leq
  2|B|^2+2|A^*|^2$. Thus, the invertibility  of $|B|^2+|A^*|^2$
  follows from that of $AB$, as wished.
\end{proof}

Let's give an explicit and non-trivial application of the previous
result.

\begin{exa}
Let $A$ be defined by $Af(x)=f'(x)$ on the domain
\[D(A)=\{f\in L^2(0,1):~f'\in L^2(0,1)\}.\]
Then $A$ is densely defined and closed (but it is not normal, see
e.g. \cite{Mortad-RCMP-2015}). Also,
\[\text{$A^*f(x)=-f'(x)$ on $D(A^*)=\{f\in L^2(0,1):~f'\in
L^2(0,1),~f(0)=f(1)=0\}$}\] so that
$|A^*|^2=AA^*f(x)=-f''(x)$
with
\[D(AA^*)=\{f\in L^2(0,1):~f''\in
L^2(0,1),~f(0)=f(1)=0\}.\] Let $V$ be the Volterra operator defined
on $L^2(0,1)$, i.e.
\[(Vf)(x)=\int_0^xf(t)dt,~f\in L^2(0,1).\]
Then (see e.g. \cite{Mortad-Oper-TH-BOOK-WSPC})
$|V|^2f=V^*Vf=\sum_{n=1}^{\infty}\lambda_n \langle f,f_n \rangle f_n$
with $(f_n)$ being the eigenvectors corresponding to the distinct
eigenvalues $\lambda_n$ of $V^*V$ (and $f\in L^2(0,1)$).

Now, neither $A$ nor $V$ is invertible. However, $A$ is right
invertible for $D(AV)=L^2(0,1)$ and
$AVf(x)=f(x)$ for $f\in L^2(0,1)$.
This means that $AV$ is positive and invertible. Therefore, the
unbounded operator
\[-\frac{d^2}{dx^2}+|V|^2\]
is invertible on the domain $D(AA^*)$ given above.
\end{exa}

Next, we pass to the invertibility of finite sums.

\begin{lem}(cf. \cite{Costara-Popa-exos-func-analysis-2003-trans-from-Russian})\label{absolute value finite sum scalar INEQ Test} Let
$(A_k)_{k=1,...,n}$ be in $B(H)$ and let $(a_k)_{k=1,...,n}$ be in
$\C$. Then
\[\left\|\sum_{k=1}^na_kA_kx\right\|^2\leq \sum_{k=1}^{n}|a_k|^2\left\langle\sum_{k=1}^{n}A_k^*A_kx,x\right\rangle\]
for all $x\in H$.
\end{lem}

\begin{proof}
Clearly,
\[\left\|\sum_{k=1}^{n}a_kA_kx\right\|\leq \sum_{k=1}^{n}\left\|a_kA_kx\right\|\leq \left(\sum_{k=1}^{n}|a_k|^2\right)^{\frac{1}{2}}\left(\sum_{k=1}^{n}\|A_kx\|^2\right)^{\frac{1}{2}}\]
for all $x\in H$.
\end{proof}

\begin{cor}
Let $(A_k)_{k=1,...,n}$ be in $B(H)$ and let $(a_k)_{k=1,...,n}$ be
in $\C$. If $\sum_{k=1}^{n}a_kA_k$ is invertible, then so is
$\sum_{k=1}^n|A_k|^2$.
\end{cor}

\begin{proof}As $\sum_{k=1}^{n}a_kA_k$ is
invertible, it is bounded below, i.e., for some $\alpha>0$ and all
$x\in H$, we have
$\left\|\sum_{k=1}^na_kA_kx\right\|\geq \alpha \|x\|$.
By Lemma \ref{absolute value finite sum scalar INEQ Test} and by the
self-adjointness of $\sum_{k=1}^{n}A_k^*A_k$, it follows that
$\sum_{k=1}^n|A_k|^2$ is invertible (given that the $a_k$ cannot all
vanish simultaneously!).
\end{proof}

\begin{thm}\label{inver sum A_k B_k THM}
Let $(A_k)_{k=1,...,n}$ and let $(B_k)_{k=1,...,n}$ be in $B(H)$. If
$\sum_{k=1}^nA_kB_k$ is positive and invertible, then
$\sum_{k=1}^n|A_k^*|^2+\sum_{k=1}^n|B_k|^2$ too is invertible.
\end{thm}

\begin{proof}
Let $x\in H$. Then
\begin{align*}
&\left\langle\sum_{k=1}^nA_kB_kx,x\right\rangle=\left|\left\langle\sum_{k=1}^nA_kB_kx,x\right\rangle\right|=\left|\sum_{k=1}^n\langle A_kB_kx,x\rangle\right|\\
&=\left|\sum_{k=1}^n\langle B_kx,A_k^*x\rangle\right|\leq
\sum_{k=1}^n|\langle B_kx,A_k^*x\rangle|\leq \sum_{k=1}^n\|B_kx\|\|A_k^*x\|\\
&\leq \sqrt{\sum_{k=1}^n\|B_kx\|^2}\sqrt{\sum_{k=1}^n\|A_k^*x\|^2}=\sqrt{\sum_{k=1}^n\langle B_k^*B_kx,x\rangle}\sqrt{\sum_{k=1}^n\langle A_kA_k^*x,x\rangle}
\end{align*}
\begin{align*}
&=\sqrt{\left\langle\sum_{k=1}^nB_k^*B_kx,x\right\rangle}\sqrt{\left\langle\sum_{k=1}^nA_kA_k^*x,x\right\rangle}\\
& \leq
\frac{1}{2}\left(\left\langle\sum_{k=1}^n|B_k|^2x,x\right\rangle+\left\langle\sum_{k=1}^n|A_k^*|^2x,x\right\rangle\right)
&=\frac{1}{2}\left\langle\sum_{k=1}^n(|B_k|^2+|A_k^*|^2)x,x\right\rangle.
\end{align*}
Since $\sum_{k=1}^nA_kB_k$ is positive and invertible, Proposition
\ref{invert positivity A B New Deh-Mortad PRO} gives the
invertibility of $\sum_{k=1}^n|A_k^*|^2+\sum_{k=1}^n|B_k|^2$, as
needed.
\end{proof}

\begin{cor}
Let $(A_k)_{k=1,...,n}$ be in $B(H)$. If $\sum_{k=1}^nA_kA^*_k$ (or
$\sum_{k=1}^nA^*_kA_k$) is invertible, then so is
$\sum_{k=1}^n|A_k^*|^2+\sum_{k=1}^n|A_k|^2$.
\end{cor}

\begin{proof}
First, remember that $AA^*_k\geq0$ for every $k$. Hence clearly
$\sum_{k=1}^nA_kA^*_k\geq0$. Now apply Theorem \ref{inver sum A_k
B_k THM} to get the desired result.
\end{proof}

\end{document}